\documentclass[runningheads]{llncs}

\usepackage[T1]{fontenc}
\def\doi#1{\href{https://doi.org/\detokenize{#1}}{\url{https://doi.org/\detokenize{#1}}}}
\usepackage{graphicx}
\usepackage{listings}
\usepackage{amsmath,amsfonts,amssymb}
\usepackage{tikz}
\usepackage{algpseudocode}
\usepackage{algorithm}
\usepackage{esvect}
\newcommand{\lvv}[1]{\reflectbox{\ensuremath{\vv{\reflectbox{\ensuremath{#1}}}}}}
\usepackage{comment}
\usepackage{mathtools}
\DeclareMathOperator\PSB {PSB}

\errorcontextlines\maxdimen

\makeatletter
\newcommand*{\algrule}[1][\algorithmicindent]{\makebox[#1][l]{\hspace*{.5em}\thealgruleextra\vrule height \thealgruleheight depth \thealgruledepth}}%
\newcommand*{\thealgruleextra}{}
\newcommand*{\thealgruleheight}{.75\baselineskip}
\newcommand*{\thealgruledepth}{.25\baselineskip}

\newcount\ALG@printindent@tempcnta
\def\ALG@printindent{%
	\ifnum \theALG@nested>0
	\ifx\ALG@text\ALG@x@notext
	\else
	\unskip
	\addvspace{-1pt}
	\ALG@printindent@tempcnta=1
	\loop
	\algrule[\csname ALG@ind@\the\ALG@printindent@tempcnta\endcsname]%
	\advance \ALG@printindent@tempcnta 1
	\ifnum \ALG@printindent@tempcnta<\numexpr\theALG@nested+1\relax
	\repeat
	\fi
	\fi
}%
\usepackage{etoolbox}
\patchcmd{\ALG@doentity}{\noindent\hskip\ALG@tlm}{\ALG@printindent}{}{\errmessage{failed to patch}}
\makeatother

\newbox\statebox
\newcommand{\myState}[1]{%
	\setbox\statebox=\vbox{#1}%
	\edef\thealgruleheight{\dimexpr \the\ht\statebox+1pt\relax}%
	\edef\thealgruledepth{\dimexpr \the\dp\statebox+1pt\relax}%
	\ifdim\thealgruleheight<.75\baselineskip
	\def\thealgruleheight{\dimexpr .75\baselineskip+1pt\relax}%
	\fi
	\ifdim\thealgruledepth<.25\baselineskip
	\def\thealgruledepth{\dimexpr .25\baselineskip+1pt\relax}%
	\fi
	\State #1%
	\def\thealgruleheight{\dimexpr .75\baselineskip+1pt\relax}%
	\def\thealgruledepth{\dimexpr .25\baselineskip+1pt\relax}%
}

\begin{document}
\title{On 1-skeleton of the polytope \\ of pyramidal tours with step-backs\thanks{Supported by the P.G. Demidov Yaroslavl State University Project VIP-016}}
%
%
\author{Andrei V. Nikolaev \orcidID{0000-0003-4705-2409}}
\authorrunning{A.V. Nikolaev}
%
\institute{P.G. Demidov Yaroslavl State University, Yaroslavl, Russia\\
\email{andrei.v.nikolaev@gmail.com}}
\maketitle              
\begin{abstract}
Pyramidal tours with step-backs are Hamiltonian tours of a special kind: the salesperson starts in city 1, then visits some cities in ascending order, reaches city $n$, and returns to city 1 visiting the remaining cities in descending order. However, in the ascending and descending direction, the order of neighboring cities can be inverted (a step-back). It is known that on pyramidal tours with step-backs the traveling salesperson problem can be solved by dynamic programming in polynomial time.

We define the polytope of pyramidal tours with step-backs $\PSB(n)$ as the convex hull of the characteristic vectors of all possible pyramidal tours with step-backs in a complete directed graph. The 1-skeleton of $\PSB(n)$ is the graph whose vertex set is the vertex set of the polytope, and the edge set is the set of geometric edges or one-dimensional faces of the polytope. 
We present a linear-time algorithm to verify vertex adjacencies in 1-skeleton of the polytope $\PSB(n)$ and estimate the diameter and the clique number of 1-skeleton: the diameter is bounded above by 4 and the clique number grows quadratically in the parameter $n$.

\end{abstract}

\section{Introduction}

The 1-\textit{skeleton} of a polytope $P$ is the graph whose vertex set is the vertex set of $P$ and edge set is the set of geometric edges or one-dimensional faces of $P$.
In this paper, we consider 3 characteristics of 1-skeleton: vertex adjacency, graph diameter, and clique number.

Two vertices of a graph $G$ are called \textit{adjacent} iff they share a common edge.
Vertex adjacency in 1-skeleton is of interest as it can be directly applied to develop simplex-like combinatorial optimization algorithms that move from one feasible solution to another along the edges of the 1-skeleton.
This class includes, for example, the blossom algorithm by Edmonds for constructing maximum matchings \cite{Edmonds1965}, set partitioning algorithm by Balas and Padberg \cite{Balas1975}, Balinski's algorithm for the assignment problem \cite{Balinski1985}, Ikura and Nemhauser's algorithm for the set packing problem \cite{Ikura1985}, etc.

The \textit{diameter} of a graph $G$ is the maximum edge distance between any pair of vertices
The study of $1$-skeleton's diameter is motivated by its relationship to the simplex-method and similar edge-following algorithms since the diameter serves as a lower bound for the number of iterations of such algorithms (see \cite{Dantzig1963,Grotchel1985}), as well as the famous Hirsch conjecture \cite{Dantzig1963,Santos2012}.

The \textit{clique number} of a graph $G$, denoted by $\omega(G)$, is the number of vertices in a maximum clique of $G$. It is known that the clique number of 1-skeleton is a lower bound for computational complexity in a class of \textit{direct-type} algorithms based on linear comparisons \cite{Bondarenko1983,Bondarenko2013}.
Besides, this characteristic is polynomial for known polynomially solvable problems and is superpolynomial for intractable problems (see, for example, \cite{Bondarenko2016,Bondarenko2017SPT,Simanchev2018}).

\section{Traveling salesperson polytope}

We consider an asymmetric traveling salesperson problem: given a complete weighted digraph $K_n=(V,E)$ (whose vertices are called \textit{cities}), it is required to find a Hamiltonian tour of minimum weight \cite{Garey1979}.
With each Hamiltonian tour $x$ in $K_{n}$ we associate a characteristic vector $\mathbf{v}(x) \in \mathbb{R}^{E}$ by the following rule:
\[
\mathbf{v}(x)_e = 
\begin{cases}
1,& \text{ if an edge } e \in E \text{ is contained in the tour } x,\\
0,& \text{ otherwise. }
\end{cases}
\]
An example of constructing a characteristic vector $\mathbf{v}(x)$ for a Hamiltonian tour $x$ is shown in Fig.~\ref{image:characteristic_vector}.

\begin{figure}[t]
	\centering
	\begin{tikzpicture}[scale=1]
	\begin{scope}[every node/.style={circle,thick,draw,inner sep=3pt}]
	\node (1) at (0,3) {1};
	\node (2) at (3,3) {2};
	\node (3) at (3,0) {3};
	\node (4) at (0,0) {4};
	\end{scope}
	\draw [->,>=stealth,thick, bend left=10] (1) edge node[above]{\scriptsize 1} (2);
	\draw [->,>=stealth, bend left=10,dashed] (2) edge node[below,pos=0.4]{\scriptsize 2} (1);
	\draw [->,>=stealth, bend left=10,dashed] (2) edge node[right]{\scriptsize 3} (3);
	\draw [->,>=stealth, bend left=10,dashed] (3) edge node[left]{\scriptsize 4} (2);
	\draw [->,>=stealth, bend right=10,dashed] (3) edge node[above,pos=0.6]{\scriptsize 5} (4);
	\draw [->,>=stealth,thick, bend right=10] (4) edge node[below]{\scriptsize 6} (3);
	\draw [->,>=stealth, bend left=10,dashed] (4) edge node[left]{\scriptsize 7} (1);
	\draw [->,>=stealth, bend left=10,dashed] (1) edge node[right]{\scriptsize 8} (4);
	\draw [->,>=stealth, bend left=10,dashed] (1) edge node[above,pos=0.3]{\scriptsize 9} (3);
	\draw [->,>=stealth,thick, bend left=10] (3) edge node[below,pos=0.3]{\scriptsize 10} (1);
	\draw [->,>=stealth, bend left=10,dashed] (4) edge node[above,pos=0.3]{\scriptsize 11} (2);
	\draw [->,>=stealth,thick, bend left=10] (2) edge node[below,pos=0.3]{\scriptsize 12} (4);
	
	\node at (1.5,-1) {$\mathbf{v}(1,2,4,3)=(1,0,0,0,0,1,0,0,0,1,0,1)$};
	\end{tikzpicture}
	\caption{An example of a characteristic vector for a Hamiltonian tour $\langle 1,2,4,3 \rangle$}
	\label{image:characteristic_vector}
\end{figure}

The polytope
\[\operatorname{ATSP}(n) = \operatorname{conv} \{\mathbf{v}(x) \ | \ x \text{ is a Hamiltonian tour in } K_n\}\]
is called \textit{the asymmetric traveling salesperson polytope}.

The \textit{symmetric traveling salesperson polytope} $\operatorname{TSP}(n)$ is defined similarly as the convex hull of characteristic vectors of all possible Hamiltonian cycles in the complete
undirected graph $K_n$.

The traveling salesperson polytope was introduced by Dantzig, Fulkerson, and Johnson in their classic work on solving the traveling salesperson problem for 49 US cities by integer linear programming \cite{Dantzig1954}.
State-of-the-art exact algorithms for the traveling salesperson problem are based on a partial description of the facets of the traveling salesperson polytope and the branch and cut method for integer linear programming \cite{Applegate2006}.

The 1-skeleton of the traveling salesperson polytope has long been the object of close attention in the field of polyhedral combinatorics.
The classic result by Papadimitriou \cite{Papadimitriou1978} states that the question of whether two vertices of the $\operatorname{ATSP}(n)$ (or $\operatorname{TSP}(n)$) are not adjacent is NP-complete. 
It is known that the graph diameter of 1-skeleton equals 2 for $\operatorname{ATSP}(n)$ \cite{Padberg1974} and is at most 4 for $\operatorname{TSP}(n)$ \cite{Rispoli1998}.
An open conjecture by Gr\"{o}tchel and Padberg states that the diameter is 2 for both polytopes \cite{Grotchel1985}.
As for the clique number of $\operatorname{ATSP}(n)$ (and $\operatorname{TSP}(n)$), Bondarenko proved that it is superpolynomial in the parameter $n$ \cite{Bondarenko1983}.
Note that, historically, the traveling salesperson polytope was the first combinatorial polytope for which both the NP-completeness of verifying the vertex non-adjacency and the superpolynomial clique number of the 1-skeleton were established.

Since vertex adjacency is a hard problem for the traveling salesperson polytope, various special cases are of interest.
In particular, Sierksma et al. \cite{Sierksma1995} studied the faces of diameter 2, Arthanari \cite{Arthanari2013} considered the pedigree polytope which provided a sufficient condition for non-adjacency in the traveling salesperson polytope, and Bondarenko et al. \cite{Bondarenko2018,Bondarenko2017} studied the polytope of pyramidal tours.

In this paper, we consider the polytope associated with Hamiltonian tours of a special kind: pyramidal tours with step-backs.

\section{Pyramidal tours}

We suppose that the cities are labeled from $1$ to $n$.
Let $\tau$ be a Hamiltonian tour.
We denote the successor of $i$-th city as $\tau (i)$, and the predecessor as $\tau^{-1} (i)$.
For any natural $k$, we denote the $k$-th successor of $i$ as $\tau^{k}(i)$, the $k$-th predecessor of $i$ as $\tau^{-k}(i)$.
The city $i$ satisfying $\tau^{-1}(i) < i$ and $\tau(i) < i$ is called a \textit{peak}.
A \textit{pyramidal tour}, introduced by Aizenshtat and Kravchuk \cite{Aizenshtat1968}, is a Hamiltonian tour with only one peak $n$.
In other words, the salesperson starts in city $1$, then visits some cities in ascending order, reaches city $n$ and returns to city $1$, visiting the remaining cities in descending order.

Enomoto, Oda, and Ota introduced a more general class of pyramidal tours with step-backs \cite{Enomoto1998}.
A \textit{step-back peak} (see Fig. \ref{Fig_step_backs}) is the city $i$, such that
\[\tau^{-1} < i, \ \tau(i) = i - 1, \ \tau^{2}(i) > i, \text{ or } \tau^{-2} > i, \ \tau^{-1}(i) = i - 1, \ \tau(i) < i.\]
A \textit{proper peak} is a peak $i$ which is not a step-back peak.
A \textit{pyramidal tour with step-backs} is a Hamiltonian tour with exactly one proper peak $n$.

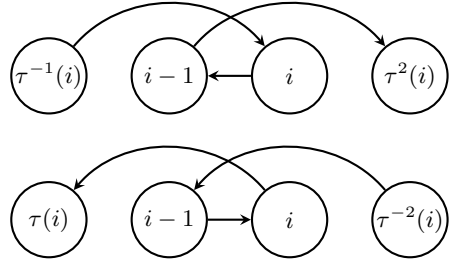
\begin{figure} [t]
	\centering
	\begin{tikzpicture}[scale=0.8]
	\begin{scope}[every node/.style={circle,thick,draw,inner sep=0pt, minimum size=1cm}]
	\node (A) at (0,0) {$\tau^{-1}(i)$};
	\node (B) at (2,0) {$i-1$};
	\node (C) at (4,0) {$i$};
	\node (D) at (6,0) {$\tau^{2}(i)$};
	
	\end{scope}
	\draw [thick,->,>=stealth] (A) edge [bend left=50] (C);
	\draw [thick,->,>=stealth] (C) edge (B);
	\draw [thick,->,>=stealth] (B) edge [bend left=50] (D);
	\end{tikzpicture}
	\\[2mm]
	\begin{tikzpicture}[scale=0.8]
	\begin{scope}[every node/.style={circle,thick,draw,inner sep=0pt, minimum size=1cm}]
	\node (A) at (0,0) {$\tau(i)$};
	\node (B) at (2,0) {$i-1$};
	\node (C) at (4,0) {$i$};
	\node (D) at (6,0) {$\tau^{-2}(i)$};
	
	\end{scope}
	\draw [thick,<-,>=stealth] (A) edge [bend left=50] (C);
	\draw [thick,<-,>=stealth] (C) edge (B);
	\draw [thick,<-,>=stealth] (B) edge [bend left=50] (D);
	\end{tikzpicture}
	
	\caption {A step-back in ascending and descending order}
	\label {Fig_step_backs}
\end{figure}

Pyramidal tours and pyramidal tours with step-backs are of interest, since, on the one hand, the minimum cost pyramidal tour (with step-backs) can be found in $O(n^2)$ time by dynamic programming, and, on the other hand, there are known restrictions on the distance matrix that guarantee the existence of an optimal tour that is pyramidal (with step-backs). See \cite{Burkard1998,Gilmore1985} for pyramidal tours, and \cite{Enomoto1998} for pyramidal tours with step-backs.


We consider a complete digraph $K_n=(V,E)$.
With each pyramidal tour (with step-backs) $x$ in $K_{n}$ we associate a characteristic vector $\mathbf{v}(x) \in \mathbb{R}^{E}$:
\[
\mathbf{v}(x)_e = 
\begin{cases}
1,& \text{ if an edge } e \in E \text{ is contained in the tour } x,\\
0,& \text{ otherwise. }
\end{cases}
\]

The polytope
\[\operatorname{PYR}(n) = \operatorname{conv} \{\mathbf{v}(x) \ | \ x \text{ is a pyramidal tour in } K_n\}\]
is called the \textit{polytope of pyramidal tours}.

The polytope
\[\operatorname{PSB}(n) = \operatorname{conv} \{\mathbf{v}(x) \ | \ x \text{ is a pyramidal tour with step-backs in } K_n\}\]
is called the \textit{polytope of pyramidal tours with step-backs}.

The polytope of pyramidal tours $\operatorname{PYR}(n)$ was introduced in \cite{Bondarenko2017} and later considered in \cite{Bondarenko2018} by Bondarenko et al. It was established that vertex adjacency in 1-skeleton of the $\operatorname{PYR}(n)$ polytope can be verified in linear time $O(n)$, the diameter of 1-skeleton equals 2, and the asymptotically exact estimate of clique number is $\Theta(n^2)$.

The polytope of pyramidal tours with step-backs was introduced in \cite{Nikolaev2019}, where a necessary and sufficient condition for vertex adjacency in the 1-skeleton of the polytope is given.
Based on this condition, we develop a linear-time algorithm to verify vertex adjacencies in the polytope $\operatorname{PSB}(n)$ and study the diameter and the clique number of the 1-skeleton.

\section{Pyramidal encoding}

Following \cite{Nikolaev2019}, we introduce a special pyramidal encoding to represent the pyramidal tours with step-backs.
With each pyramidal tour with step-backs $x$ in $K_n$ we associate a vector $\mathbf{x}^{0,1,sb}$ of length $n-2$, each coordinate corresponds to a city from $2$ to $n-1$, by the following rule:
\begin{gather*}
\mathbf{x}^{0,1,sb}_{i} =
\begin{cases}
1, &\text{ if } i \text{ is visited by } x \text{ in ascending order},\\
\lvv{ 1 \ 1}, &\text{ if } i \text{ is a step-back peak in ascending order},\\
0, &\text{ if } i \text{ is visited by } x \text{ in descending order},\\
\vv{0 \ 0}, &\text{ if } i \text{ is a step-back peak in descending order}.
\end{cases}
\end{gather*}
Note that a step-back on $i$ also involves the previous coordinate $i-1$.
An example of a pyramidal tour with step-backs and the corresponding encoding vector $\mathbf{x}^{0,1,sb}$ is shown in Fig.~\ref{Fig_encoding}.

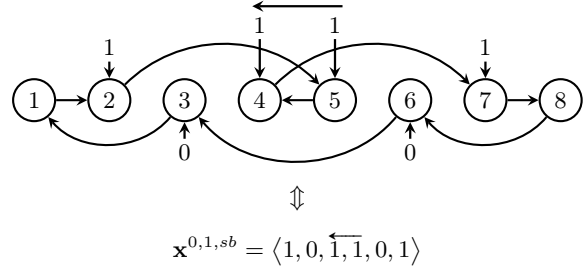
\begin{figure}[t]
	\centering
	\begin{tikzpicture}[scale=1.0]
	\begin{scope}[every node/.style={circle,thick,draw,inner sep=3pt}]
	\node (A) at (0,0) {1};
	\node (B) at (1,0) {2};
	\node (C) at (2,0) {3};
	\node (D) at (3,0) {4};
	\node (E) at (4,0) {5};
	\node (F) at (5,0) {6};
	\node (G) at (6,0) {7};
	\node (H) at (7,0) {8};
	\end{scope}
	\draw [thick,->,>=stealth] (A) edge (B);
	\draw [thick,->,>=stealth] (B) edge [bend left=45] (E);
	\draw [thick,->,>=stealth] (E) edge (D);
	\draw [thick,->,>=stealth] (D) edge [bend left=45] (G);
	\draw [thick,->,>=stealth] (G) edge (H);
	\draw [thick,->,>=stealth] (H) edge [bend left=50] (F);
	\draw [thick,->,>=stealth] (F) edge [bend left=50] (C);
	\draw [thick,->,>=stealth] (C) edge [bend left=50] (A);
	
	\node at (3.5,-1.3) {$\Updownarrow$};
	\node at (3.5,-2) {$\mathbf{x}^{0,1,sb} = \left\langle 1,0,\lvv{1,1},0,1 \right\rangle$};
	
	\begin{scope}[every node/.style={circle,inner sep=1pt}]
	\node (J) at (2,-0.7) {0};
	\node (K) at (3,1.0) {1};
	\node (L) at (4,1.0) {1};
	\node (M) at (5,-0.7) {0};
	\node (N) at (6,0.7) {1};
	\node (O) at (1,0.7) {1};
	\end{scope}
	
	\draw [->,>=stealth,line width=0.3mm] (J) edge (C);
	\draw [->,>=stealth,line width=0.3mm] (K) edge (D);
	\draw [->,>=stealth,line width=0.3mm] (L) edge (E);
	\draw [->,>=stealth,line width=0.3mm] (M) edge (F);
	\draw [->,>=stealth,line width=0.3mm] (N) edge (G);
	\draw [->,>=stealth,line width=0.3mm] (O) edge (B);
	
	\draw [<-,>=stealth,line width=0.3mm] (2.9,1.25) -- (4.1, 1.25);
	\end{tikzpicture}
	\caption {An example of a tour and the corresponding pyramidal encoding}
	\label {Fig_encoding}
\end{figure}

We denote by $\mathbf{x}^{0,1,sb}_{[i,j]}$ a fragment of encoding on coordinates from $i$ to $j$.
The superscript indicates what we consider in the encoding: descending order ($0$),
ascending order ($1$), or step-backs ($sb$).
For example, $\mathbf{x}^{1,sb}_{[i,j]}$ means a fragment of the encoding only in ascending order taking into account step-backs; $\mathbf{x}^{0,1}_{[i,j]}$ -- a fragment of the encoding disregarding step-backs, etc.

\section{Vertex adjacency}

We consider 12 blocks of the following form (a wavy line means that the corresponding coordinate can either contain a step-back or not):
\begin{gather*}
U_{11} = \left\langle \begin{matrix} 1 \\ 1 \end{matrix} \right\rangle,\
U_{00} = \left\langle \begin{matrix} 0 \\ 0 \end{matrix} \right\rangle,\
U_{1111} = \left\langle \begin{matrix} \lvv {1 \ \ 1} \\ \lvv {1 \ \ 1} \end{matrix} \right\rangle,\
U_{0000} = \left\langle \begin{matrix} \vv {0 \ \ 0} \\ \vv {0 \ \ 0} \end{matrix} \right\rangle,\\
L_{1110} = \left\langle \begin{matrix} \lvv {1 \ \ 1} \\ 1 \ \ \tilde{0} \end{matrix} \right\rangle,\
L_{1011} = \left\langle \begin{matrix} 1 \ \ \tilde{0} \\ \lvv {1 \ \ 1} \end{matrix} \right\rangle,\
L_{0001} = \left\langle \begin{matrix} \vv {0 \ \ 0} \\ 0 \ \ \tilde{1} \end{matrix} \right\rangle,\
L_{0100} = \left\langle \begin{matrix} 0 \ \ \tilde{1} \\ \vv {0 \ \ 0} \end{matrix} \right\rangle,\\
R_{1101} = \left\langle \begin{matrix} \lvv {1 \ \ 1} \\ \tilde{0} \ \ 1 \end{matrix} \right\rangle,\
R_{0111} = \left\langle \begin{matrix} \tilde{0} \ \ 1 \\ \lvv {1 \ \ 1} \end{matrix} \right\rangle,\
R_{0010} = \left\langle \begin{matrix} \vv {0 \ \ 0} \\ \tilde{1} \ \ 0 \end{matrix} \right\rangle,\
R_{1000} = \left\langle \begin{matrix} \tilde{1} \ \ 0 \\ \vv {0 \ \ 0} \end{matrix} \right\rangle.
\end{gather*}

\begin{theorem}[Nikolaev \cite{Nikolaev2019}] \label{Theorem:adjacency}
	Vertices $\mathbf{v}(x)$ and $\mathbf{v}(y)$ of the polytope $\operatorname{PSB}(n)$ are not adjacent if and only if the following conditions are satisfied.
	\begin{itemize}
		\item There exists a city $i$ (called a \textit{left block}) such that the tours $x$ and $y$ on the coordinate $i$ (coordinates $i$ and $i+1$ for double blocks) in the pyramidal encoding have the form of $U,L$, or $i = 1$.
		
		\item There exists a city $j$ (called a \textit{right block}) such that the tours $x$ and $y$ on the coordinate $j$ (coordinates $j-1$ and $j$ for double blocks) in the pyramidal encoding have the form of $U,R$, or $j = n$.
		
		We denote by $i_a$ the first city after the left block: $i_a = i+1 $ for single blocks and $i_a = i+2$ for double blocks.
		We denote by $j_b$ the last city before the right block: $j_b = i-1$ for single blocks and $j_b = j-2$ for double blocks.
		
		Two blocks cut the encoding of the tours into three parts: the left (less than $i_a$), the central (from $i_a$ to $j_b$), and the right (larger than $j_b$).
		
		\item In the central part, the coordinates of $\mathbf{x}^{0,1}$ and $\mathbf{y}^{0,1}$ completely coincide: $\mathbf{x}^{0,1}_{[i_a,j_b]} = \mathbf{y}^{0,1}_{[i_a,j_b]}$.
		
		We say that two tours
		\begin{itemize}
			\item differ in the left part if $\mathbf{x}^{0,1,sb}_{[1,i_a-1]} \neq \mathbf{y}^{0,1,sb}_{[1,i_a-1]}$,
			\item differ in the right part if $\mathbf{x}^{0,1,sb}_{[j_b+1,n]} \neq \mathbf{y}^{0,1,sb}_{[j_b+1,n]}$,
			\item differ in the central part in ascending order if $\mathbf{x}^{1,sb}_{[i_a,j_b]} \neq \mathbf{y}^{1,sb}_{[i_a,j_b]}$,
			\item differ in the central part in descending order if $\mathbf{x}^{0,sb}_{[i_a,j_b]} \neq \mathbf{y}^{0,sb}_{[i_a,j_b]}$.
		\end{itemize}
		
		The remaining conditions are divided into four cases depending on the values of $\mathbf{x}^{0,1}_i$ and $\mathbf{x}^{0,1}_j$.
		
		\begin{enumerate}
			\item If $\mathbf{x}^{0,1}_{i} = \mathbf{x}^{0,1}_{j} = 1$, then the tours differ
			\begin{itemize}
				\item in the central part in ascending order;
				\item in the left part, or in the central part in descending order, or in the right part.
			\end{itemize}
			
			\item If $\mathbf{x}^{0,1}_{i} = \mathbf{x}^{0,1}_{j} = 0$, then the tours differ
			\begin{itemize}
				\item in the central part in descending order;
				\item in the left part, or in the central part in ascending order, or in the right part.
			\end{itemize}
			
			\item If $\mathbf{x}^{0,1}_{i} = 1, \mathbf{x}^{0,1}_{j} = 0$, then the tours differ	
			\begin{itemize}
				\item in the central part in ascending order or in the right part;
				\item in the central part in descending order or in the left part.
			\end{itemize}
			
			\item If $\mathbf{x}^{0,1}_{i} = 0, \mathbf{x}^{0,1}_{j} = 1$, then the tours differ
			
			\begin{itemize}
				\item in the central part in descending order or in the right part;
				\item in the central part in ascending order or in the left part.
			\end{itemize}
		\end{enumerate}
		
		Cities $1$ and $n$ can be considered in the encoding as visited in ascending or descending order, if required.
	\end{itemize}
\end{theorem}

The idea of sufficient conditions is that if from the edges of the tours $x$ and $y$ we can assemble two complementary pyramidal tours with step-backs $z$ and $t$, then the segment $[\mathbf{v}(x),\mathbf {v}(y)]$ intersects with the segment $[\mathbf{v}(z),\mathbf{v}(t)]$, and the corresponding vertices of the polytope $\PSB(n)$ are not adjacent (see \cite{Nikolaev2019}).

The examples of the first and third sufficient conditions for non-adjacency are shown in Fig.~\ref{Fig_sufficient} (edges of $x$ are solid, edges of $y$ are dashed, left and right blocks in pyramidal encodings are highlighted with dashed boxes).

\begin{figure} [t]
	\centering
	\begin{tikzpicture}[scale=0.8]
	
	\node at (-0.7,0) {$\mathbf{x}$};
	\node at (0,0) {$ \langle $};
	\node at (1,0) {1};
	\node at (2,0) {1};
	\node at (3,0) {1};
	\node at (4,0) {1};
	\node at (5,0) {1};
	\node at (6,0) {$ \rangle $};
	
	\draw [<-,>=stealth] (2.85,0.25) -- (4.1,0.25);
	
	\node at (-0.7,-0.75) {$\mathbf{y}$};
	\node at (0,-0.75) {$ \langle $};
	\node at (1,-0.75) {0};
	\node at (2,-0.75) {1};
	\node at (3,-0.75) {1};
	\node at (4,-0.75) {1};
	\node at (5,-0.75) {1};
	\node at (6,-0.75) {$ \rangle $};
	
	\draw [dashed] (1.75,0.25) -- (2.25,0.25) -- (2.25,-1) -- (1.75,-1) -- cycle;
	\draw [dashed] (4.75,0.25) -- (5.25,0.25) -- (5.25, -1) -- (4.75,-1) -- cycle;

	\begin{scope}[yshift=-2cm]
	\begin{scope}[every node/.style={circle,thick,draw,inner sep=2.4pt}]
	\node (1) at (0,0) {1};
	\node (2) at (1,0) {2};
	\node (3) at (2,0) {3};
	\node (4) at (3,0) {4};
	\node (5) at (4,0) {5};
	\node (6) at (5,0) {6};
	\node (7) at (6,0) {7};
	\end{scope}
	\draw [thick,->,>=stealth] (1) edge (2);
	\draw [thick,->,>=stealth] (2) edge (3);
	\draw [thick,->,>=stealth] (3) edge [bend left=50] (5);	
	\draw [thick,->,>=stealth] (5) edge (4);
	\draw [thick,->,>=stealth] (4) edge [bend left=50] (6);	
	\draw [thick,->,>=stealth] (6) edge (7);
	\draw [thick,->,>=stealth] (7) edge [bend left=30] (1);
	\draw (-0.7, 0) node{\textit{x}};
	\end{scope}
	
	\begin{scope}[yshift=-4cm]
	\begin{scope}[every node/.style={circle,thick,draw,inner sep=2.4pt}]
	\node (1) at (0,0) {1};
	\node (2) at (1,0) {2};
	\node (3) at (2,0) {3};
	\node (4) at (3,0) {4};
	\node (5) at (4,0) {5};
	\node (6) at (5,0) {6};
	\node (7) at (6,0) {7};
	\end{scope}
	\draw [thick,->,>=stealth,dashed] (1) edge [bend left=50] (3);
	\draw [thick,->,>=stealth,dashed] (3) edge (4);
	\draw [thick,->,>=stealth,dashed] (4) edge (5);	
	\draw [thick,->,>=stealth,dashed] (5) edge (6);	
	\draw [thick,->,>=stealth,dashed] (6) edge (7);
	\draw [thick,->,>=stealth,dashed] (7) edge [bend left=35] (2);		
	\draw [thick,->,>=stealth,dashed] (2) edge (1);
	\draw (-0.7, 0) node{\textit{y}};		
	\end{scope}
	
	\begin{scope}[yshift=-6cm]
	\begin{scope}[every node/.style={circle,thick,draw,inner sep=2.4pt}]
	\node (1) at (0,0) {1};
	\node (2) at (1,0) {2};
	\node (3) at (2,0) {3};
	\node (4) at (3,0) {4};
	\node (5) at (4,0) {5};
	\node (6) at (5,0) {6};
	\node (7) at (6,0) {7};
	\end{scope}
	\draw [thick,->,>=stealth] (1) edge (2);
	\draw [thick,->,>=stealth] (2) edge (3);
	\draw [thick,->,>=stealth,dashed] (3) edge (4);	
	\draw [thick,->,>=stealth,dashed] (4) edge (5);
	\draw [thick,->,>=stealth] (5) edge (6);	
	\draw [thick,->,>=stealth] (6) edge (7);
	\draw [thick,->,>=stealth] (7) edge [bend left=30] (1);	
	\draw (-0.7, 0) node{\textit{z}};		
	\end{scope}
	
	\begin{scope}[yshift=-8cm]
	\begin{scope}[every node/.style={circle,thick,draw,inner sep=2.4pt}]
	\node (1) at (0,0) {1};
	\node (2) at (1,0) {2};
	\node (3) at (2,0) {3};
	\node (4) at (3,0) {4};
	\node (5) at (4,0) {5};
	\node (6) at (5,0) {6};
	\node (7) at (6,0) {7};
	\end{scope}
	\draw [thick,->,>=stealth,dashed] (1) edge [bend left=50] (3);
	\draw [thick,->,>=stealth] (3) edge [bend left=50] (5);
	\draw [thick,->,>=stealth] (5) edge (4);	
	\draw [thick,->,>=stealth] (4) edge [bend left=50] (6);
	\draw [thick,->,>=stealth,dashed] (6) edge (7);
	\draw [thick,->,>=stealth,dashed] (7) edge [bend left=35] (2);		
	\draw [thick,->,>=stealth,dashed] (2) edge (1);
	\draw (-0.7, 0) node{\textit{t}};		
	\end{scope}
	
	\begin{scope}[xshift=8cm]
	\node at (-0.7,0) {$\mathbf{x}$};
	\node at (0,0) {$ \langle $};
	\node at (1,0) {1};
	\node at (2,0) {0};
	\node at (3,0) {0};
	\node at (4,0) {1};
	\node at (5,0) {0};
	\node at (6,0) {$ \rangle $};	
	
	\node at (-0.7,-0.75) {$\mathbf{y}$};
	\node at (0,-0.75) {$ \langle $};
	\node at (1,-0.75) {1};
	\node at (2,-0.75) {0};
	\node at (3,-0.75) {0};
	\node at (4,-0.75) {0};
	\node at (5,-0.75) {0};
	\node at (6,-0.75) {$ \rangle $};
	
	\draw [->,>=stealth] (1.9,0.25) -- (3.1,0.25);
	\draw [->,>=stealth] (3.9,-0.5) -- (5.1,-0.5);
	
	\draw [dashed] (0.75,0.25) -- (1.25,0.25) -- (1.25,-1) -- (0.75,-1) -- cycle;
	\draw [dashed] (3.75,0.25) -- (5.25,0.25) -- (5.25, -1) -- (3.75,-1) -- cycle;	
	
	\begin{scope}[yshift=-2cm]
	\begin{scope}[every node/.style={circle,thick,draw,inner sep=2.4pt}]
	\node (1) at (0,0) {1};
	\node (2) at (1,0) {2};
	\node (3) at (2,0) {3};
	\node (4) at (3,0) {4};
	\node (5) at (4,0) {5};
	\node (6) at (5,0) {6};
	\node (7) at (6,0) {7};
	\end{scope}
	\draw [thick,->,>=stealth] (1) edge (2);
	\draw [thick,->,>=stealth] (2) edge [bend left=45] (5);
	\draw [thick,->,>=stealth] (5) edge [bend left=50] (7);
	\draw [thick,->,>=stealth] (7) edge (6);
	\draw [thick,->,>=stealth] (6) edge [bend left=45] (3);
	\draw [thick,->,>=stealth] (3) edge (4);
	\draw [thick,->,>=stealth] (4) edge [bend left=45] (1);
	\draw (-0.7, 0) node{\textit{x}};			
	\end{scope}
	
	\begin{scope}[yshift=-4cm]
	\begin{scope}[every node/.style={circle,thick,draw,inner sep=2.4pt}]
	\node (1) at (0,0) {1};
	\node (2) at (1,0) {2};
	\node (3) at (2,0) {3};
	\node (4) at (3,0) {4};
	\node (5) at (4,0) {5};
	\node (6) at (5,0) {6};
	\node (7) at (6,0) {7};
	\end{scope}
	\draw [thick,->,>=stealth,dashed] (1) edge (2);
	\draw [thick,->,>=stealth,dashed] (2) edge [bend left=35] (7);
	\draw [thick,->,>=stealth,dashed] (7) edge [bend left=50] (5);	
	\draw [thick,->,>=stealth,dashed] (5) edge (6);	
	\draw [thick,->,>=stealth,dashed] (6) edge [bend left=50] (4);
	\draw [thick,->,>=stealth,dashed] (4) edge (3);	
	\draw [thick,->,>=stealth,dashed] (3) edge [bend left=50] (1);
	\draw (-0.7, 0) node{\textit{y}};		
	\end{scope}
	
	\begin{scope}[yshift=-6cm]
	\begin{scope}[every node/.style={circle,thick,draw,inner sep=2.4pt}]
	\node (1) at (0,0) {1};
	\node (2) at (1,0) {2};
	\node (3) at (2,0) {3};
	\node (4) at (3,0) {4};
	\node (5) at (4,0) {5};
	\node (6) at (5,0) {6};
	\node (7) at (6,0) {7};
	\end{scope}
	\draw [thick,->,>=stealth] (1) edge (2);
	\draw [thick,->,>=stealth,dashed] (2) edge [bend left=35] (7);
	\draw [thick,->,>=stealth,dashed] (7) edge [bend left=50] (5);	
	\draw [thick,->,>=stealth,dashed] (5) edge (6);	
	\draw [thick,->,>=stealth] (6) edge [bend left=45] (3);
	\draw [thick,->,>=stealth] (3) edge (4);
	\draw [thick,->,>=stealth] (4) edge [bend left=45] (1);
	\draw (-0.7, 0) node{\textit{z}};		
	\end{scope}
	
	\begin{scope}[yshift=-8cm]
	\begin{scope}[every node/.style={circle,thick,draw,inner sep=2.4pt}]
	\node (1) at (0,0) {1};
	\node (2) at (1,0) {2};
	\node (3) at (2,0) {3};
	\node (4) at (3,0) {4};
	\node (5) at (4,0) {5};
	\node (6) at (5,0) {6};
	\node (7) at (6,0) {7};
	\end{scope}
	\draw [thick,->,>=stealth,dashed] (1) edge (2);
	\draw [thick,->,>=stealth] (2) edge [bend left=45] (5);
	\draw [thick,->,>=stealth] (5) edge [bend left=50] (7);
	\draw [thick,->,>=stealth] (7) edge (6);
	\draw [thick,->,>=stealth,dashed] (6) edge [bend left=50] (4);
	\draw [thick,->,>=stealth,dashed] (4) edge (3);	
	\draw [thick,->,>=stealth,dashed] (3) edge [bend left=50] (1);
	\draw (-0.7, 0) node{\textit{t}};		
	\end{scope}
	\end{scope}
	\end{tikzpicture}	
	\caption{Examples of first and third sufficient conditions}
	\label{Fig_sufficient}
\end{figure}
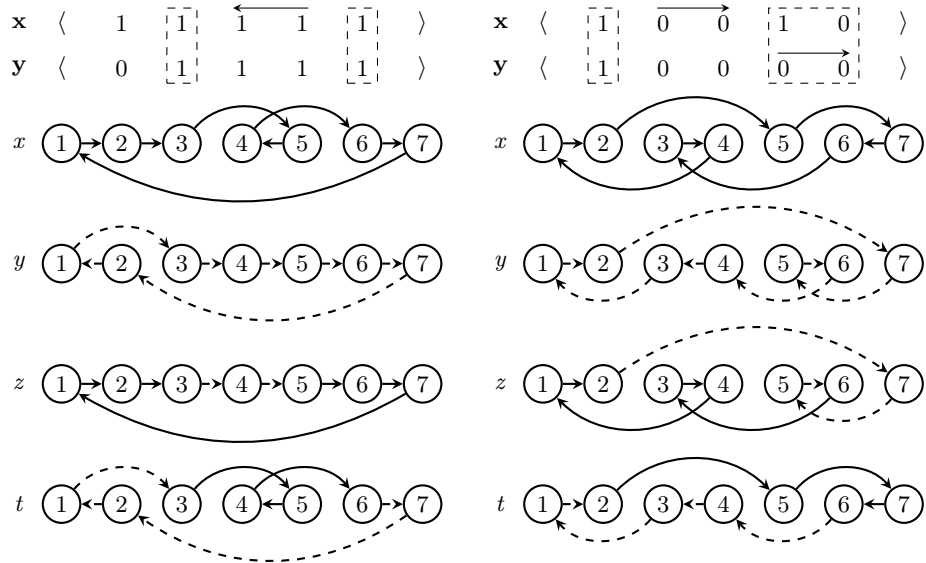

In \cite{Nikolaev2019} it was proved that the necessary and sufficient condition for non-adjacency of Theorem~\ref{Theorem:adjacency} can be verified by exhaustive search in $O(n^3)$ time. We improve this estimate by introducing a linear-time algorithm.

\begin{theorem}
	The question of whether two vertices of the polytope $\operatorname{PSB}(n)$ are adjacent can be verified in linear time $O(n)$.	
\end{theorem}

\begin{proof}
	We consider two pyramidal tours with step-backs $x$ and $y$, and the corresponding vertices $\mathbf{v}(x)$ and $\mathbf{v}(y)$ of the polytope $\operatorname{PSB}(n)$.	
	Each of the four sufficient non-adjacency conditions of Theorem~\ref{Theorem:adjacency} can be verified in a single pass through the pyramidal encodings $\mathbf{x}^{0,1,sb}$ and $\mathbf{y}^{0,1,sb}$, when we sequentially find the left block, the right block, and check additional conditions.	
	The pseudo-code to construct a tour $z$ and verify the first and second sufficient conditions is given in the Algorithm~\ref{Algorithm:adjacency}.
	The other two sufficient conditions are verified similarly.
\end{proof}

\begin{algorithm}[t]
	\caption{Verifying 1st and 2nd sufficient conditions for non-adjacency}
	\label{Algorithm:adjacency}
\begin{algorithmic}[0]
	\Procedure{NonAdjacencyTest}{$\mathbf{x},\mathbf{y},n$}
		\State $LBlock \gets \mathbf{TRUE}$ 	\Comment{Consider the city 1 as a left block}
		\State $RBlock, zNotx, zNoty \gets \mathbf{FALSE}$
		\For {$i \gets 2$ \textbf{to} $n-1$}
			\If{$LBlock = \mathbf{TRUE}$ \textbf{and} $RBlock = \mathbf{FALSE}$}		\Comment{Central part}
				\If{$\mathbf{z}_i$ ($\mathbf{y}_i$) is different from $\mathbf{x}_i$}
					\State $zNotx \gets \mathbf{TRUE}$		\Comment{$z$ visits $i$ by the edges of $y$}
				\EndIf
				\If{$zNotx = \mathbf{TRUE}$ and we found $U$ or $R$ block}
					\State $RBlock \gets \mathbf{TRUE}$		\Comment{Starting the right part}
				\EndIf
				\If{the conditions of the central part are violated}
					\State $LBlock,zNotx \gets \mathbf{FALSE}$		\Comment{Return to the left part}
				\EndIf
			\EndIf	
			\If{$LBlock = \mathbf{FALSE}$}		\Comment{Left part}
				\If{$\mathbf{z}_i$ ($\mathbf{x}_i$) is different from $\mathbf{y}_i$} 	\Comment{$z$ visits $i$ by the edges of $x$}
					\State $zNoty \gets \mathbf{TRUE}$
				\EndIf
				\If{we found $U$ or $L$ block}
					\State $LBlock \gets \mathbf{TRUE}$		\Comment{Starting the central part}
				\EndIf
			\EndIf
			\If{$RBlock = \mathbf{TRUE}$}		\Comment{Right part}
				\If{$\mathbf{z}_i$ ($\mathbf{x}_i$) is different from $\mathbf{y}_i$}		\Comment{$z$ visits $i$ by the edges of $x$}
					\State $zNoty \gets \mathbf{TRUE}$
				\EndIf
			\EndIf
		\EndFor
		\State $RBlock \gets \mathbf{TRUE}$ 	\Comment{Consider the city $n$ as a right block}
		\If{$LBlock,RBlock,zNotx,zNoty = \mathbf{TRUE}$}			
			\State 	\Return 1st/2nd sufficient condition for non-adjacency is satisfied
		\Else
			\State \Return 1st/2nd sufficient condition for non-adjacency is not satisfied
		\EndIf
	\EndProcedure		
\end{algorithmic}
\end{algorithm}

\section{Graph diameter and clique number}

Based on Theorem~\ref{Theorem:adjacency}, we estimate the diameter of 1-skeleton of $\PSB (n)$.

\begin{theorem}
	The diameter of 1-skeleton of $\PSB(n)$ is bounded above by 4.
\end{theorem}

\begin{proof}
The idea is as follows. 
For an arbitrary pyramidal tour with step-backs $x$ we construct a pyramidal tour $\hat{x}$ where 
\begin{equation}\label{Equation:pyramidal_adjacent}
\hat{\mathbf{x}}^{0,1}_i = 
\begin{cases}
0,& \text{ if } i \text{ is a part of step-back in } x\\
1,& \text{ otherwise.}
\end{cases}
\end{equation}
For instance:
\begin{align*}
&\mathbf{x}: \langle \ \ 1 \ \ \lvv{1 \ \ 1} \ \ 1 \ \ 0 \ \ 0 \ \ 1 \ \ \vv{0 \ \ 0} \ \ 1 \rangle,\\
&\hat{\mathbf{x}}: \langle \ \ 1 \ \ 0 \ \ 0 \ \ 1 \ \ 1 \ \ 1 \ \ 1 \ \ 0 \ \ 0 \ \ 1 \rangle.
\end{align*}

By construction, the encodings of the tours $x$ and $\hat{x}$ can only contain blocks $U_{11}$, which restricts us to the first sufficient condition of Theorem~\ref{Theorem:adjacency}. 
However, by (\ref{Equation:pyramidal_adjacent}) the tours $x$ and $\hat{x}$ cannot differ in the central part in the ascending order. Hence, by Theorem~\ref{Theorem:adjacency}, the vertices $\mathbf{v}(x)$ and $\mathbf{v}(\hat{x})$ are adjacent. And for any pyramidal tour $\hat{x}$, the vertex $\mathbf{v}(\hat{x})$ is adjacent to the vertices corresponding to the tours $\langle 1,1,\ldots,1 \rangle$ and $\langle 0,0,\ldots,0 \rangle$ (see \cite{Bondarenko2018}).
Thus, between any pair of vertices of the polytope $\PSB(n)$ we can construct a path of no more than 4 edges. The corresponding scheme is shown in Fig.~\ref{Fig:diameter}.
\end{proof}

\begin{figure} [t]
	\centering
\begin{tikzpicture}[scale=0.95]

\node [circle,thick,draw,inner sep=3pt, label=below:\footnotesize $\begin{matrix} \text{pyramidal tour} \\ \text{with step-backs } x \end{matrix}$] (PSB1) at (0,0) {};
\node [circle,thick,draw,inner sep=3pt,label=above:\footnotesize $\text{pyramidal tour } \hat{x}$] (P1) at (2.5,0) {};
\node [circle,thick,draw,inner sep=3pt,label=below:\footnotesize $\begin{matrix} \text{pyramidal tour} \\ \left\langle 1\,1\,1\, \ldots \, 1 \right\rangle  \end{matrix}$] (1-0) at (5,0) {};
\node [circle,thick,draw,inner sep=3pt,label=above:\footnotesize $\text{pyramidal tour } \hat{y}$]  (P2) at (7.5,0) {};
\node [circle,thick,draw,inner sep=3pt,label=below:\footnotesize $\begin{matrix} \text{pyramidal tour} \\ \text{with step-backs } y \end{matrix}$] (PSB2) at (10,0) {};

\draw [thick] (PSB1) -- (P1) -- (1-0) -- (P2) -- (PSB2);
\end{tikzpicture}
	\caption{Path of length 4 between an arbitrary pair of $\PSB(n)$ vertices $\mathbf{v}(x)$ and $\mathbf{v}(y)$}
\label{Fig:diameter}
\end{figure}
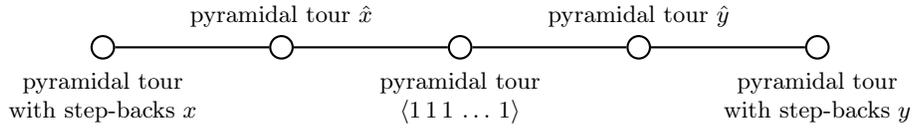

Now we apply the necessary and sufficient condition of Theorem~\ref{Theorem:adjacency} to estimate the clique number of the 1-skeleton of the polytope $\PSB(n)$.

\begin{theorem}
	The clique number of 1-skeleton of the polytope $\PSB (n)$ is quadratic in the parameter $n$:
	\begin{equation} \label{Equation:clique_number}
		\omega(\PSB (n)) = \Theta \left( n^2\right).
	\end{equation}

\end{theorem}

\begin{proof}
	\textit{Upper bound.}
	Let $Y^{\mathbf{v}}$ be a set of pairwise adjacent vertices of $\PSB(n)$, and $Y$ be the set of corresponding pyramidal tours with step-backs. Let us estimate the cardinality of $Y$.
	
	\textit{Step 1.}
	Consider the pyramidal encodings of the tours.
	We call a tour $x \in Y$ \textit{unique with respect to a pair of neighboring coordinates $k,k+1$} if
	\[\forall y \in Y \backslash \{x\}:\ \mathbf{x}^{0,1,sb}_{\left[k,k+1\right]}  \neq \mathbf{y}^{0,1,sb}_{\left[k,k+1\right]}.\]
	A pair of neighboring coordinates in the pyramidal encoding can take 18 different values. Hence the number of tours in $Y$ that are unique in a pair of neighboring coordinates does not exceed $18(n-3)$.
	We construct the set $W$ by excluding from $Y$ all tours that are unique in a pair of neighboring coordinates.

	\textit{Step 2.}
	Let a pyramidal tour with step-backs in a pair of neighboring cities $k$ and $k+1$ visit one of the cities in ascending order, and the other in descending order.
	Since the city may or may not be a part of step-back, we get 8 possible types of pyramidal encodings at coordinates $k$ and $k+1$:
	\begin{equation} \label{Equation:01-segments}
	\begin{split}
	\left\langle 1 \ \ 0 \right\rangle, \ \ \left\langle 1 \ \ \lvv{0 \ \ 0} \right\rangle, \ \  \left\langle \vv{1 \ \ 1} \ \ 0 \right\rangle, \ \ \left\langle \vv{1 \ \ 1} \ \ \lvv{0 \ \ 0} \right\rangle,\\
	\left\langle 0 \ \ 1 \right\rangle, \ \ \left\langle 0 \ \ \vv{1 \ \ 1} \right\rangle, \ \  \left\langle \lvv{0 \ \ 0} \ \ 1 \right\rangle, \ \ \left\langle \lvv{0 \ \ 0} \ \ \vv{1 \ \ 1} \right\rangle.
	\end{split}
	\end{equation}
	We call such sections of encoding \textit{$0 \slash 1$-segments}.
	
	We consider a pyramidal tour with step-backs $x \in W$ with $0 \slash 1$-segment at the coordinates $k,k+1$.
	By construction, we have excluded from $W$ all tours that are unique in a pair of neighboring coordinates, so there exists a tour $y \in W \backslash \{x\}$ with the same $0 \slash 1$-segment at coordinates $k,k+1$, i.e. $\mathbf{x}^{0,1,sb}_{\left[k,k+1\right]} = \mathbf{y}^{0,1,sb}_{\left[k,k+1\right]}$.
	
	Thus, pyramidal encodings of tours $x$ and $y$ on coordinates $k,k+1$ have the form of a pair of blocks $U$, where the coordinate $k$ is in the left block and $k+1$ is in the right block.
	Since the corresponding vertices $\mathbf{v}(x)$ and $\mathbf{v}(y)$ of the polytope $\PSB(n)$ are adjacent, by Theorem~\ref{Theorem:adjacency}, the encodings of the tours $x$ and $y$ coincide either in the left part ($\mathbf{x}^{0,1,sb}_{\left[1,k\right]} = \mathbf{y}^{0,1,sb}_{\left[1,k\right]}$), or in the right part ($\mathbf{x}^{0,1,sb}_{\left[k+1,n\right]} = \mathbf{y}^{0,1,sb}_{\left[k+1,n\right]}$).
	For instance:
	\begin{align*}
	\mathbf{x}: \langle \ \  \boxed{1 \ \ \vv{0 \ \ 0} \ \ 1} \ \ \overbracket{0 \ \ 1}^{k,k+1} \ \ \vv{0 \ \ 0} \ \ 0 \ \ 1 \ \ \rangle,\\
	\mathbf{y}: \langle \ \ \boxed{1 \ \ \vv{0 \ \ 0} \ \ 1} \ \ \underbracket{ 0 \ \ 1}_{k,k+1} \ \ 1 \ \ 0 \ \ \lvv{1 \ \ 1} \ \ \rangle.
	\end{align*}
	
	Note that for any subset of tours in $W$ with a common $0\slash 1$-segment, the coinciding parts of the encoding are on the same side of the segment. Indeed, suppose that three tours $x,y,z \in W$ have the same $0\slash 1$-segment on the cities $k,k+1$, but the coinciding parts of the encodings are on different sides of the segment.
	Without loss of generality, let
	\begin{gather*}
	\mathbf{x}^{0,1,sb}_{\left[1,k\right]} = \mathbf{y}^{0,1,sb}_{\left[1,k\right]} \text{ and } \mathbf{x}^{0,1,sb}_{\left[k+1,n\right]} = \mathbf{z}^{0,1,sb}_{\left[k+1,n\right]},\\
	\mathbf{x}: \langle \ \ \boxed{\ \ \ \ L \ \ \ \ } \ \ 0 \ \ 1 \ \ \boxed{\ \ \ \ R \ \ \ \ } \ \ \rangle,\\
	\mathbf{y}: \langle \ \ \boxed{\ \ \ \ L \ \ \ \ } \ \ 0 \ \ 1 \ \ \boxed{\ \ \ \ \phantom{R} \ \ \ \ } \ \ \rangle,\\
	\mathbf{z}: \langle \ \ \boxed{\ \ \ \ \phantom{L} \ \ \ \ } \ \ 0 \ \ 1 \ \ \boxed{\ \ \ \ R \ \ \ \ } \ \ \rangle.
	\end{gather*}
	However, since the corresponding vertices $\mathbf{v}(y)$ and $\mathbf{v}(z)$ of the polytope $\PSB(n)$ are adjacent, by Theorem~\ref{Theorem:adjacency}, the pyramidal encodings of the tours $y$ and $z$ also coincide either in the left part: $\mathbf{y}^{0,1,sb}_{\left[1,k\right]} = \mathbf{z}^{0,1,sb}_{\left[1,k\right]}$ (in this case $x = z$), or in the right part: $\mathbf{y}^{0,1,sb}_{\left[k+1,n\right]} = \mathbf{z}^{0,1,sb}_{\left[k+1 ,n\right]}$ (in this case $x = y$). We got a contradiction.
	
	Thus, for any pyramidal tour with step-backs from the set $W$, all $0\slash 1$-segments can be divided into \textit{left segments} (for which tours with a common segment coincide in the left part) and \textit{right segments} (for which tours coincide in the right part).
	
	Let us show that the left and right $0 \slash 1$-segments in tours from $W$ are ordered, i.e. if some tour $x \in W$ contains a left $0\slash 1$-segment on cities $k,k+1$ and a right $0\slash 1$-segment on cities $s,s+1$, then $k < s$.
	
	Assume that $k > s$. Consider a tour $y \in W$ that shares a left $0\slash 1$-segment on $k,k+1$ with $x$, then $\mathbf{x}^{0,1,sb}_{\left[1 ,k\right]} = \mathbf{y}^{0,1,sb}_{\left[1,k\right]}$.
	Hence the tours $x$ and $y$ coincide on the cities $s,s+1$ since $s < k$, and have common right $0\slash 1$-segments:
	\begin{align*}
	&\mathbf{x}: \langle \ \boxed{* \ \ * \ \ 0 \ \ 1 \ \ * \ \ *} \ \ \overset{L}{\overbracket{0 \ \ 1}} \ \ * \ \ * \ \ \rangle,\\
	&\mathbf{y}: \langle \ \ * \ \ * \ \ \underset{R}{\underbracket{0 \ \ 1}} \ \ \boxed{* \ \ * \ \ 0 \ \ 1 \ \ * \ \ *} \ \rangle.
	\end{align*}
	It remains to note that
	\[\mathbf{x}^{0,1,sb}_{\left[1,k\right]} = \mathbf{y}^{0,1,sb}_{\left[1,k\right]} \text{ and } \mathbf{x}^{0,1,sb}_{\left[s,n\right]} = \mathbf{y}^{0,1,sb}_{\left[s,n\right]},\]
	hence, $x = y$, a contradiction.

	\textit{Step $3$}.
	Consider a pyramidal tour with step-backs $x \in W$.
	We choose the left $0\slash 1$-segment $L_{\max}$ at the largest coordinates $i,i+1$. If there is no such segment, then we set $i=1$.
	We choose the right $0\slash 1$-segment $R_{\min}$ at the smallest coordinates $j-1,j$. If there is no such segment, then we set $j = n$.
	
	By construction, $i \leq j$ and cities from $i+1$ to $j-1$ are visited in the same direction.
	Let's call the part of pyramidal encoding $\left[i+1,j-1\right]$ a \textit{$0$-sequence} if the cities from $i+1$ to $j-1$ are visited in the descending order, and \textit{$1$-sequence} if the cities are visited in ascending order.
	
	Since the tours in $W$ coincide to the left of the common left $0\slash 1$-segment and to the right of the common right $0\slash 1$-segment, each pyramidal tour with step-backs $x \in W$ corresponds to a unique $0$-sequence or $1$-sequence. For example:
	\[
	\langle \ \ \underbracket{1 \ \ 0}_{L} \ \ 0 \ \ \underbracket{\vv{0 \ \ 0} \ \ 1}_{L} \ \ 1 \ \ \rlap{$\underbracket{\phantom{1 \ \ 0}}_{L_{\max}}$} 1 \ \ \overbracket{0 \ \ \vv{0 \ \ 0} \ \ 0 \ \ \vv{0 \ \ 0} \ \ \rlap{$\underbracket{\phantom{0 \ \ 1}}_{R_{\min}}$} 0}^{0\text{-sequence}} \ \ 1 \ \ \underbracket{\vv{1 \ \ 1} \ \ 0}_{R} \ \ \vv{0 \ \ 0} \ \ \rangle.
	\]
	
	\textit{Step $4$}.
	We consider in $W$ the subset of all pyramidal tours with step-backs that have a $0$-sequence starting at position $i$. Let's denote this subset as $W^{0}_i$.
	
	We consider in the set $W^{0}_i$ all tours containing at least one unique $0$-coordinate inside the $0$-sequence in the pyramidal encoding. There are $n-i$ possible positions of the unique $0$-coordinate, which can take $3$ different values: start of a step-back, end of a step-back, not a step-back. Thus, the total number of such unique tours in $W^{0}_i$ does not exceed $3(n-i)=O(n)$. Let us construct the set $\bar{W}^{0}_i$ by excluding from $W^{0}_i$ all tours with unique coordinates in the $0$-sequence.

	Consider some tour $x \in \bar{W}^{0}_{i}$. By construction, for any coordinate $s$ within the $0$-sequence, there exists a second tour $y \in \bar{W}^{0}_{i}$ such that $\mathbf{x}^{0,sb}_{s} = \mathbf{y}^{0,sb}_s$ and $s$ belongs to the $0$-sequence of $y$.
	Then the tours $x$ and $y$ on the coordinates $i-1$ (segment $L_{\max}$) form a block $U_{11}$ (or $U_{1111}$), and on the coordinate $s$ -- one of the blocks $U_{00}$ or $U_{0000}$.
	Therefore, by Theorem~\ref{Theorem:adjacency}, the pyramidal encodings of the tours $x$ and $y$ coincide either in the central part between $i-1$ and $s$ in descending order ($\mathbf{x}^{0,sb}_{\left[i, s\right]} = \mathbf{y}^{0,sb}_{\left[i,s\right]}$) or to the right of $s$ ($\mathbf{x}^{0,1,sb}_{\left [s,n\right]} = \mathbf{y}^{0,1,sb}_{\left[s,n\right]}$).
	For instance:
	\begin{align*}
	&\mathbf{x}: \langle \ \boxed{0 \ \ \lvv{1 \ \ 1}} \ \overbracket{1}^{i-1} \ \boxed{0 \ \ \vv{0 \ \ 0}} \ \ \overbracket{0}^{s} \ \ 0 \ \ 0 \ \ 1 \ \ 0 \ \ \lvv{1 \ \ 1} \ \ \rangle,\\
	&\mathbf{y}: \langle \ \boxed{0 \ \ \lvv{1 \ \ 1}} \ \underbracket{1}_{i-1} \ \boxed{0 \ \ \vv{0 \ \ 0}} \ \ \underbracket{0}_{s} \ \ \vv{0 \ \ 0} \ \ 0 \ \ 0 \ \  1 \ \ 1 \ \ \rangle.
	\end{align*}
	Otherwise, the corresponding vertices $\mathbf{v}(x)$ and $\mathbf{v}(y)$ of the polytope $\PSB(n)$ are not adjacent.
	
	Further reasoning completely repeats similar ones for common $0 \slash 1$-segments. For any subset of tours in $\bar{W}^{0}_{i}$ with a common $0$-coordinate, the coinciding parts of the encoding must be on the same side of the coordinate.
	This allows us to divide all $0$-sequence coordinates into \textit{left $0$-coordinates} (tours with a common coordinate coincide on the left side) and \textit{right $0$-coordinates} (for which tours coincide to the right of the common coordinate). Note also that all coordinates are ordered, i.e. if $s$ is the left 0-coordinate and $t$ is the right 0-coordinate, then $s < t$:
	\[
	\langle \ \boxed{0 \ \ \lvv{1 \ \ 1}} \ \overbracket{1}^{i-1} \ \underbracket{\overbracket{0 \ \ \vv{0 \ \ 0} \ \ 0}^{L} \ \ \overbracket{\vv{0 \ \ 0} \ \ 0}^{R}}_{0\text{-sequence}} \ \ 1 \ \ \vv{1 \ \ 1}  \ \ 0 \ \ \rangle.
	\]
	
	We denote by $\bar{W}^{0}_{i,s}$ the subset of all tours from $\bar{W}^{0}_i$ for which $s$ is the number of the largest left $0$-coordinate, and $s+1$ is the number of the smallest right $0$-coordinate.
	
	The key idea is that the coordinates $s$ and $s+1$ of any tour from $\bar{W}^{0}_{i,s}$ can take only 2 values: $\langle 0 \rangle $ and $\langle \lvv{0 \ \ 0} \rangle$. 
	Moreover, the values of this pair of coordinates uniquely determine the tour since tours with the same left coordinate coincide to the left and with the same right coordinate -- to the right.
	Thus, the subset $\bar{W}^{0}_{i,s}$ contains at most 4 tours.
	
	And now we rise back to the original set of pairwise adjacent tours.
	Firstly,
	\[\bar{W}^{0}_i = \bigcup^{n-1}_{s = i+1} \bar{W}^{0}_{i,s},\]
	whence, with the excluded tours with unique $0$-coordinates, $|W^0_i| = O(n)$.
	Similarly, we get that $|W^1_i| = O(n)$.
	
	Secondly,
	\[ W = \left( \bigcup^{n-1}_{i = 2} W^{0}_{i}\right)  \cup \left( \bigcup^{n-1}_{i = 2} W^{1}_{i}\right) ,\]
	whence we get that $|W| = O(n^2)$.
	
	Finally, returning the excluded tours with unique pairs of neighboring coordinates, we arrive at the upper bound $|Y| = |W| + O(n) = O(n^2)$. Thus,
	\[\omega(\PSB(n)) = O(n^2).\]

	\textit{Lower bound.}
	We construct an example of the set $Z^{\mathbf{v}}$ of pairwise adjacent vertices of the polytope $\PSB(n)$ such that $|Z^{\mathbf{v}}| = \Omega (n^2)$.
	Let $n$ be even. With each pair of integers $k,s$ such that $0 \leq k,s \leq \frac{n-2}{2}$ we associate a pyramidal tour with step-backs $x(k,s) \in Z$ such that
	\begin{itemize}
		\item $\forall i$ $(2 \leq i \leq k+1)$: $\mathbf{x}^{0,1,sb}_i = 1$;
		\item $\forall j$ $(n-s \leq j \leq n-1)$: $\mathbf{x}^{0,1,sb}_j = 1$;
		\item all other coordinates of $\mathbf{x}^{0,1,sb}$ are equal to $0$.
	\end{itemize}

	The tours from the set $Z$ do not contain step-backs, therefore, they cannot differ in the central part between the left and right blocks. Moreover, if for two tours $x,y \in Z$ along some coordinate $i$: $\mathbf{x}^{0,1,sb}_i = \mathbf{y}^{0,1,sb}_i = 1$, then
	\begin{itemize}
		\item if $i \leq \frac{n}{2}$, then the encodings coincide in the left part (i.e.  $\mathbf{x}^{0,1,sb}_{[1,i]} = \mathbf{y}^{0,1,sb}_{[1,i]}$);
		\item if $i \geq \frac{n}{2}+1$, then the pyramidal encodings coincide in the right part (i.e. $\mathbf{x}^{0,1,sb}_{[i,n]} = \mathbf{y}^{0,1,sb}_{[i,n]}$).
	\end{itemize}
	Hence, by Theorem~\ref{Theorem:adjacency}, the corresponding vertices $\mathbf{v}(x)$ and $\mathbf{v}(y)$ of the polytope $\PSB(n)$ are adjacent.
	
	The case of odd $n$ is reduced to an even case, it suffices to fix one of the cities in all tours from $Z$ in ascending or descending order. Thus,
	\[ \omega(\PSB(n)) \geq	|Z| = \left\lfloor \frac{n}{2} \right\rfloor^2 = \Omega (n^2). \]
	
	An example of a set $Z$ for $n = 8$ is given in Table~\ref{Table:Lower_bound}.
	
	\begin{table}[t]
		\centering
		\caption{An example of pyramidal tours with step-backs with pairwise adjacent vertices in the polytope $\PSB(8)$}
		\label{Table:Lower_bound}
	\begin{tabular}{cccc}
		~$\left\langle 0 \ 0 \ 0 \ | \ 0 \ 0 \ 0 \right\rangle$~ & ~$\left\langle 0 \ 0 \ 0 \ | \ 0 \ 0 \ 1 \right\rangle$~ & ~$\left\langle 0 \ 0 \ 0 \ | \ 0 \ 1 \ 1 \right\rangle$~ & ~$\left\langle 0 \ 0 \ 0 \ | \ 1 \ 1 \ 1 \right\rangle$~ \\
		$\left\langle 1 \ 0 \ 0 \ | \ 0 \ 0 \ 0 \right\rangle$ & $\left\langle 1 \ 0 \ 0 \ | \ 0 \ 0 \ 1 \right\rangle$ & $\left\langle 1 \ 0 \ 0 \ | \ 0 \ 1 \ 1 \right\rangle$ & $\left\langle 1 \ 0 \ 0 \ | \ 1 \ 1 \ 1 \right\rangle$\\
		$\left\langle 1 \ 1 \ 0 \ | \ 0 \ 0 \ 0 \right\rangle$ & $\left\langle 1 \ 1 \ 0 \ | \ 0 \ 0 \ 1 \right\rangle$ & $\left\langle 1 \ 1 \ 0 \ | \ 0 \ 1 \ 1 \right\rangle$ & $\left\langle 1 \ 1 \ 0 \ | \ 1 \ 1 \ 1 \right\rangle$\\
		$\left\langle 1 \ 1 \ 1 \ | \ 0 \ 0 \ 0 \right\rangle$ & $\left\langle 1 \ 1 \ 1 \ | \ 0 \ 0 \ 1 \right\rangle$ & $\left\langle 1 \ 1 \ 1 \ | \ 0 \ 1 \ 1 \right\rangle$ & $\left\langle 1 \ 1 \ 1 \ | \ 1 \ 1 \ 1 \right\rangle$\\
	\end{tabular}
	\end{table}

Combining the upper and lower bounds, we obtain the desired asymptotically exact quadratic estimate (\ref{Equation:clique_number}).
\end{proof}

\section{Conclusion}

The results of the research are summarized in Table~\ref{Table:Results}.

\begin{table}[h]
	\centering
	\caption{Properties of the $\operatorname{ATSP}(n)$, $\operatorname{PYR}(n)$, and $\PSB(n)$ polytopes}
	\label{Table:Results}
	\resizebox{\textwidth}{!}{
	\begin{tabular}{c|c|c|c|c}
		& \begin{tabular}{@{}c@{}}~Complexity of~ \\ TSP problem\end{tabular} & \begin{tabular}{@{}c@{}}~Vertex adjacency~ \\ in 1-skeleton\end{tabular} & \begin{tabular}{@{}c@{}}Diameter \\ of 1-skeleton\end{tabular} & 		\begin{tabular}{@{}c@{}}~Clique number~ \\ of 1-skeleton\end{tabular}  \\ \hline
		\begin{tabular}{@{}c@{}}Hamiltonian cycles \\ $\operatorname{ATSP}(n)$ and $\operatorname{TSP}(n)$\end{tabular} & NP-hard & co-NP-complete & \begin{tabular}{@{}c@{}} 2 for $\operatorname{ATSP}(n)$ \\ $\leq 4$ for $\operatorname{TSP}(n)$\end{tabular} & $\Omega\left( 2^{\left(\sqrt {n} - 9\right) \slash 2}\right) $ \\ \hline
		\begin{tabular}{@{}c@{}}Pyramidal tours~  \\ $\operatorname{PYR}(n)$ \end{tabular} & $O(n^2)$ & $O(n)$ & 2 & $\Theta(n^2)$ \\ \hline
		~\begin{tabular}{@{}c@{}}Pyramidal tours with \\ step-backs $\PSB(n)$\end{tabular}~ & $O(n^2)$ & $O(n)$ & $\leq 4$ & $\Theta(n^2)$
	\end{tabular}
}
\end{table}

We have considered several versions of the traveling salesperson problem and the associated combinatorial polytopes.

The general traveling salesperson problem is NP-hard \cite{Garey1979}. The question of whether two vertices of the traveling salesperson polytope are not adjacent is NP-complete \cite{Papadimitriou1978}. 
The clique number of 1-skeleton of the polytope $\operatorname{ATSP}(n)$ is superpolynomial in $n$ \cite{Bondarenko1983}. 

On the other hand, the traveling salesperson problem on pyramidal tours and pyramidal tours with step-backs is solvable in polynomial time by dynamic programming \cite{Enomoto1998,Gilmore1985}. 
The vertex adjacency for the polytopes $\operatorname{PYR}(n)$ and $\operatorname{PSB}(n)$ can be verified in linear time $O(n)$. The clique numbers of 1-skeletons are quadratic in $n$.

Thus, the properties of 1-skeletons of the polytopes associated with the traveling salesperson problem directly correlate with the complexity of the problem itself.


%
%
%
 \bibliographystyle{splncs04}
 \bibliography{Nikolaev_MOTOR2022}

\end{document}